\documentclass[11pt]{amsart}
\usepackage{geometry}
\usepackage{amsmath, amssymb}
 \usepackage{amsmath,amscd}
\usepackage{amsfonts}
\usepackage{mathrsfs}
\usepackage[arrow,matrix,curve,cmtip,ps]{xy}
\usepackage{setspace}\usepackage[dvipsnames]{xcolor}

\usepackage{pb-diagram} 
\usepackage{pb-xy}
\usepackage{graphicx}

\usepackage{amsthm}

\usepackage{color}

\usepackage{xcolor}

\allowdisplaybreaks

\newtheorem*{rep@theorem}{\rep@title}
\newcommand{\newreptheorem}[2]{%
\newenvironment{rep#1}[1]{%
 \def\rep@title{#2 \ref{##1}}%
 \begin{rep@theorem}}%
 {\end{rep@theorem}}}
\makeatother

\newtheorem{proposition}{Proposition}[section]
\newtheorem{theorem}[proposition]{Theorem}

\newtheorem{lemma}[proposition]{Lemma}

\newreptheorem{theorem}{Theorem}
\newreptheorem{proposition}{Proposition}

\newtheorem*{theorem*}{Theorem}
\newtheorem*{proposition*}{Proposition}
\theoremstyle{remark}

\newtheorem{definition}[proposition]{Definition}

\newtheorem{question}[proposition]{Question}

\newcommand{\lk}{\operatorname{lk}}

\newcommand{\M}{\mathcal{M}}
\newcommand{\N}{\mathbb{N}}
\newcommand{\Z}{\mathbb{Z}}

\newcommand{\bdry}{\partial}

\newcommand{\onto}{\twoheadrightarrow}

\newcommand{\GCD}{\operatorname{GCD}}

\newcommand{\rank}{\operatorname{rank}}

\newcommand{\sign}{\operatorname{sign}}

\newcommand{\Sum}{\displaystyle\sum}
\newcommand{\Oplus}{\displaystyle\bigoplus}

\newcommand{\pref}[1]{(\ref{#1})}

\begin{document}

\title{On the indeterminacy of Milnor's triple linking number}

\author{Jonah Amundsen}
\address{Department of Mathematics, University of Wisconsin-Eau Claire, Hibbard Humanities Hall 508,  Eau Claire WI 54702-4004}
\email{amundsjj3573@uwec.edu}

\author{Eric Anderson}
\address{Department of Mathematics, University of Wisconsin-Eau Claire, Hibbard Humanities Hall 508,  Eau Claire WI 54702-4004}
\email{andersew1951@uwec.edu}

\author{Christopher William Davis}
\address{Department of Mathematics, University of Wisconsin-Eau Claire, Hibbard Humanities Hall 508,  Eau Claire WI 54702-4004}
\email{daviscw@uwec.edu}

\date{\today}

\subjclass[2010]{}


\begin{abstract} 
In the 1950's Milnor defined a family of higher order  invariants generalizing the linking number.  Even the first of these new invariants, the triple linking number, has received and fruitful study since its inception.  In the case that $L$ has vanishing pairwise linking numbers, this triple linking number gives an integer valued invariant.  When the linking numbers fail to vanish, this invariant is only well-defined modulo their greatest common divisor.  In recent work Davis-Nagel-Orson-Powell produce a single invariant called the \emph{total triple linking number} refining the triple linking number and taking values in an abelian group called the \emph{total Milnor quotient}.  They present examples for which this quotient is nontrivial even though none of the individual triple linking numbers are defined.  As a consequence, the total triple linking number carries more information than do the classical triple linking numbers.  The goal of the present paper is to compute this group and show that when $L$ is a link of at least six components it is non-trivial.  Thus, this total triple linking number carries information for every $(n\ge 6)$-component link, even though the classical triple linking numbers often carry no information.
\end{abstract}

\maketitle

\section{Introduction}

In the 1950's Milnor \cite{Milnor} introduced a family of invariants generalizing the classical pairwise linking number.  In this paper we  interest ourselves with first of these invariants, the so-called triple linking number.  These associate to an $n$-component link $L$ and a list of three distinct indices, $1\le i,j,k \le n$, an integer $\mu_{ijk}(L)\in \Z$ measuring how three components $L_i$, $L_j$, and  $L_k$ of $L$ interact.  When the pairwise linking numbers $\lk(L_i,L_j)$, $\lk(L_i,L_k)$,  and $\lk(L_j,L_k)$ all vanish $\mu_{ijk}(L)\in \Z$ is well-defined.  Otherwise it is only well-defined modulo the greatest common divisor (GCD) of these linking numbers.  In particular if $\lk(L_i,L_j) = 1$ then $\mu_{ijk}(L)$ takes values in the trivial group $\Z/1$, and so carries no information about the link.

  In \cite{DNOP}, the third author together with Nagel, Orson, and Powell find that if one  gathers all $n\choose 3$ triple linking numbers into a particular quotient, $\M$, of $\Z^{n\choose 3}$ then one gets a refined invariant, called the \emph{total triple linking number} $\mu(L)\in \M$ which often recovers strictly more information than do the individual triple linking numbers.  This quotient is called the \emph{total Milnor quotient} and depends only on the various pairwise linking numbers.  As a proof that this collection of triple linking numbers carries more information they exhibit a pair of 4-component links $L$ and $L'$ which have pairwise linking numbers equal to $1$, so that none of the classical triple linking numbers carry any information.  They compute that $\M\cong \Z$ and that in $\M$, $\mu(L)\neq \mu(L')$.  See \cite[Example 5.9]{DNOP}.  Thus, even when the individual triple linking numbers carry no information, it is possible that $\mu(L)$ does.  Moreover, they show that when $n\ge 9$, $\M$ is an abelian group with positive rank regardless of the pairwise linking numbers \cite[Remark 5.10]{DNOP}.  Thus, while there exist links of arbitrarily many components for which none of the triple linking numbers are defined, the total triple linking number always carries information, provided the number of components is at least $9$.  Our main result, Theorem~\ref{thm: nontrivial M}, lowers this threshold from $9$ components to $6$.

Before we state this theorem we need some notation.  The linking matrix for an $n$-component link $L$ is the $n\times n$ matrix, denoted $\Lambda$, with zeros and the main diagonal and $(i,j)$-entry $\Lambda_{ij}=\lk(L_i,L_j)$.  Any symmetric matrix with zeros on the main diagonal can be realized as a linking matrix and so we will call such a matrix a \emph{linking matrix} without reference to any particular link.  The total Milnor quotient $\M$ defined in \cite{DNOP} depends only on the linking matrix.  We write $\M(\Lambda)$ when we need to emphasize the dependence of $\M$ on $\Lambda$.  We recall the precise definition of the total Milnor quotient and the total triple linking number in Section \ref{sect:background}.  

\begin{theorem}\label{thm: nontrivial M}
Let $\Lambda$ be an $n\times n$ linking matrix with $n\ge 6$.  The resulting total Milnor quotient $\M(\Lambda)$ is non-trivial.  
\end{theorem}

In order to prove Theorem~\ref{thm: nontrivial M}, we will demonstrate a lower bound on the rank of $\M$.

\begin{theorem}\label{thm: M rank}
Let $n\ge 6$, $\Lambda$ be an $n\times n$ linking matrix, and $\M(\Lambda)$ be the resulting total Milnor quotient.  Then $\rank(\M(\Lambda))\ge\frac{n^3-9n^2+20n-6}{6}$. 
\end{theorem}

Applying Theorem \ref{thm: M rank} in the case that $n=6$ one sees that if $L$ is a 6-component link, then $\rank(\M)\ge 1$.  Moreover, standard reveal that the lower bound $\frac{n^3-9n^2+20n-6}{6}\ge 1$ when $n\ge 6$, so that $\M$ is nontrivial for all links of at least six components.  
  
  Of course, the nontriviality of $\M$ is only important if it can be used to distinguish links.  Implicit in the work of \cite{DNOP} is the idea that if $\M$ is not the trivial group, then every element of $\M$ is realized by a link.  The following theorem makes that explicit.

\begin{theorem}\label{thm: linking numbers}
Let  $\Lambda$ be a linking matrix and $\M(\Lambda)$ be the resulting total Milnor quotient.  For every element $m \in \M(\Lambda)$ there exists a link, $L$, with linking matrix $\Lambda$ such that $\mu(L) = m$.
\end{theorem}

Combining Theorems \ref{thm: nontrivial M} and \ref{thm: linking numbers}, for any $n\times n$ linking matrix $\Lambda$ with $n\ge 6$ there exist links with linking matrix $\Lambda$ which are distinguished by their total triple linking numbers.   We show that this does not follow when $n\le 5$ by producing a $5\times 5$ linking matrix for which $\M=0$.  
\begin{theorem}\label{thm: trivial M}
Let $\Lambda = \left[\begin{array}{ccccc}0&1&1&0&0\\1&0&1&1&0\\1&1&0&1&1\\0&1&1&0&1\\0&0&1&1&0
\end{array}\right]$.  The resulting total Milnor quotient $\M(\Lambda)$ is trivial.
\end{theorem}

One important application of triple linking numbers and their refinement $\mu(L)\in \M$ is to the study of surface systems bounded by a link.  A surface system for a link $L$ is a collection of Seifert surfaces for the various components of $L$ which intersect transversely.  See Figure~\ref{fig: surface system} for some examples.  Interestingly, while any two knots admit homeomorphic surfaces, two links can fail to admit homeomorphic surface systems.  Indeed according to \cite{DNOP} pairwise linking numbers, together with total triple linking number form a complete set of obstructions to links bounding homeomorphic surface systems.

\begin{figure}
\begin{picture}(270,100)
\put(0,10){\includegraphics[height=.2\textwidth]{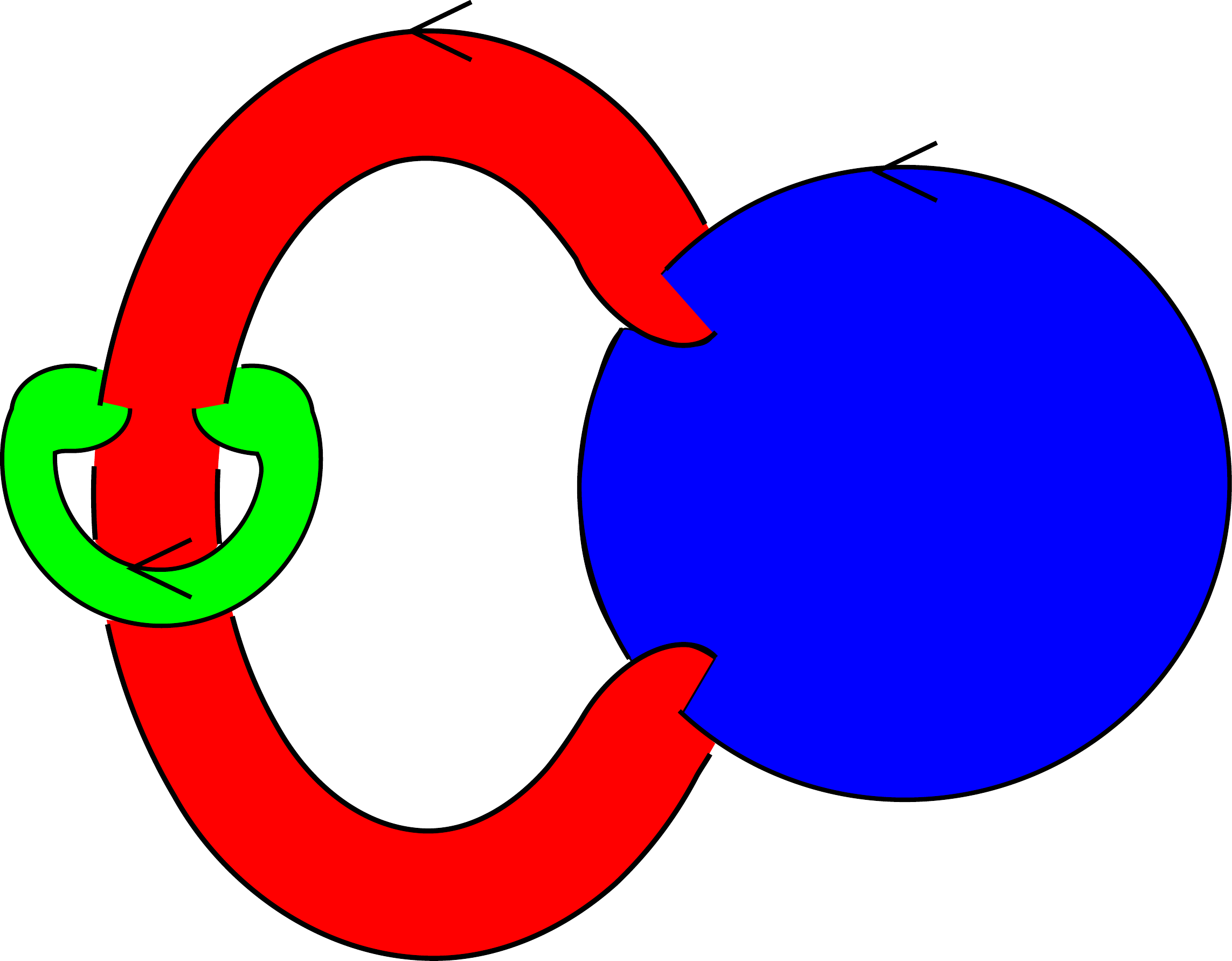}}
\put(140,10){\includegraphics[height=.2\textwidth]{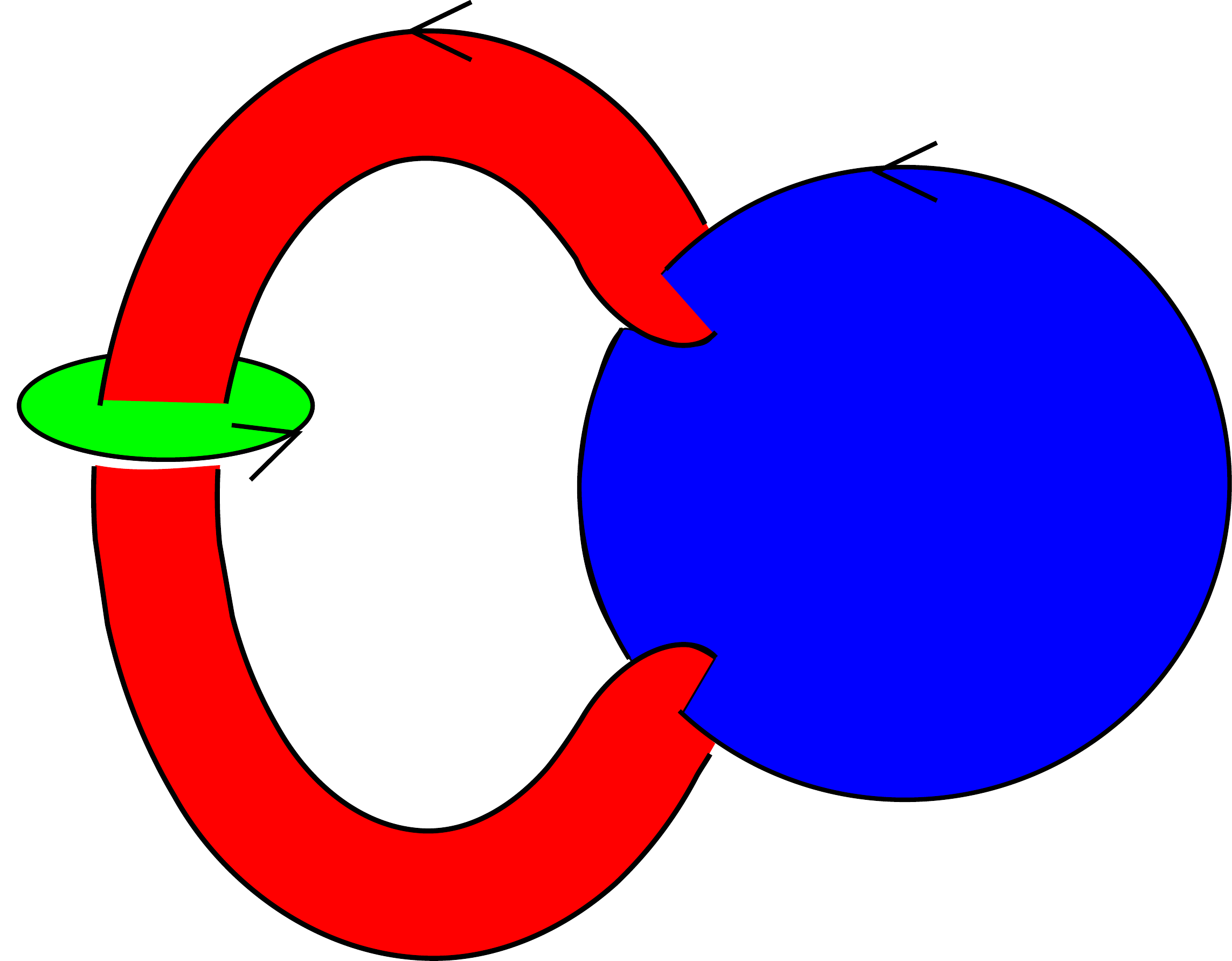}}
\end{picture}
\caption{A pair of surface systems for the Boromean rings. }
\label{fig: surface system}
\end{figure}

\begin{theorem}[Theorem 1.1 of \cite{DNOP}]\label{thm:DNOP}
Let $L$ and $L'$ be links with linking matrix $\Lambda$.  Then $L$ and $L'$ admit homeomorphic Surface systems if and only if $\mu(L)=\mu(L')$ in $\M(\Lambda)$.  
\end{theorem}

Combining Theorems \ref{thm: trivial M} and \ref{thm:DNOP} we see that if $L$ and $L'$ are 5-component links with linking matrix appearing in Theorem \ref{thm: trivial M} then $\M=0$ and so $\mu(L) = \mu(L')$  as they both live in the trivial group.   Thus, they admit  homeomorphic surface systems.  Two such links appear in Figure \ref{fig: trivial M}.  

\begin{figure}[t]
\begin{picture}(200,60)
\put(0,44){\includegraphics[angle=90, width=.2\textwidth]{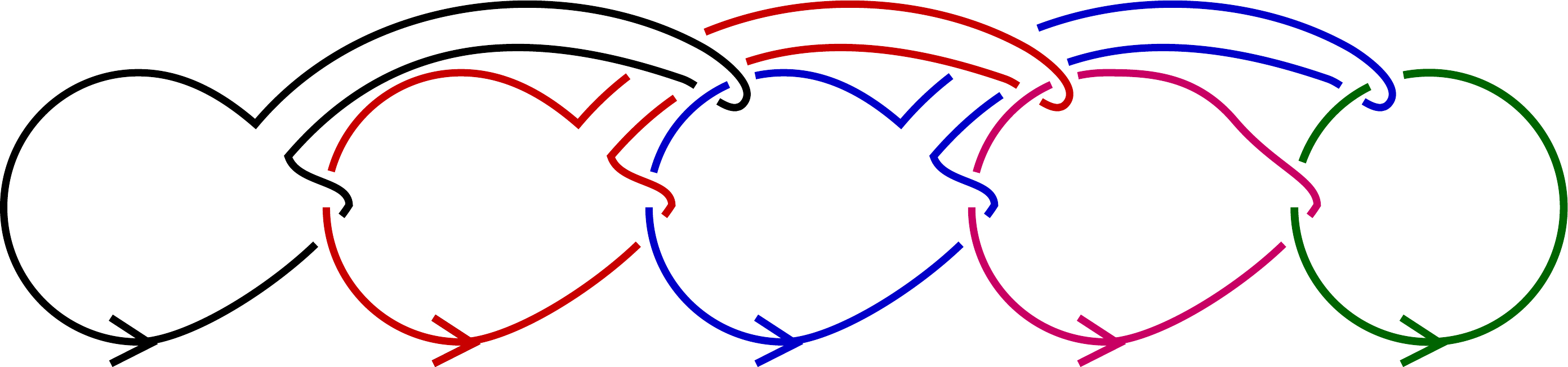}}
\put(110,0){\includegraphics[angle=90, width=.2\textwidth]{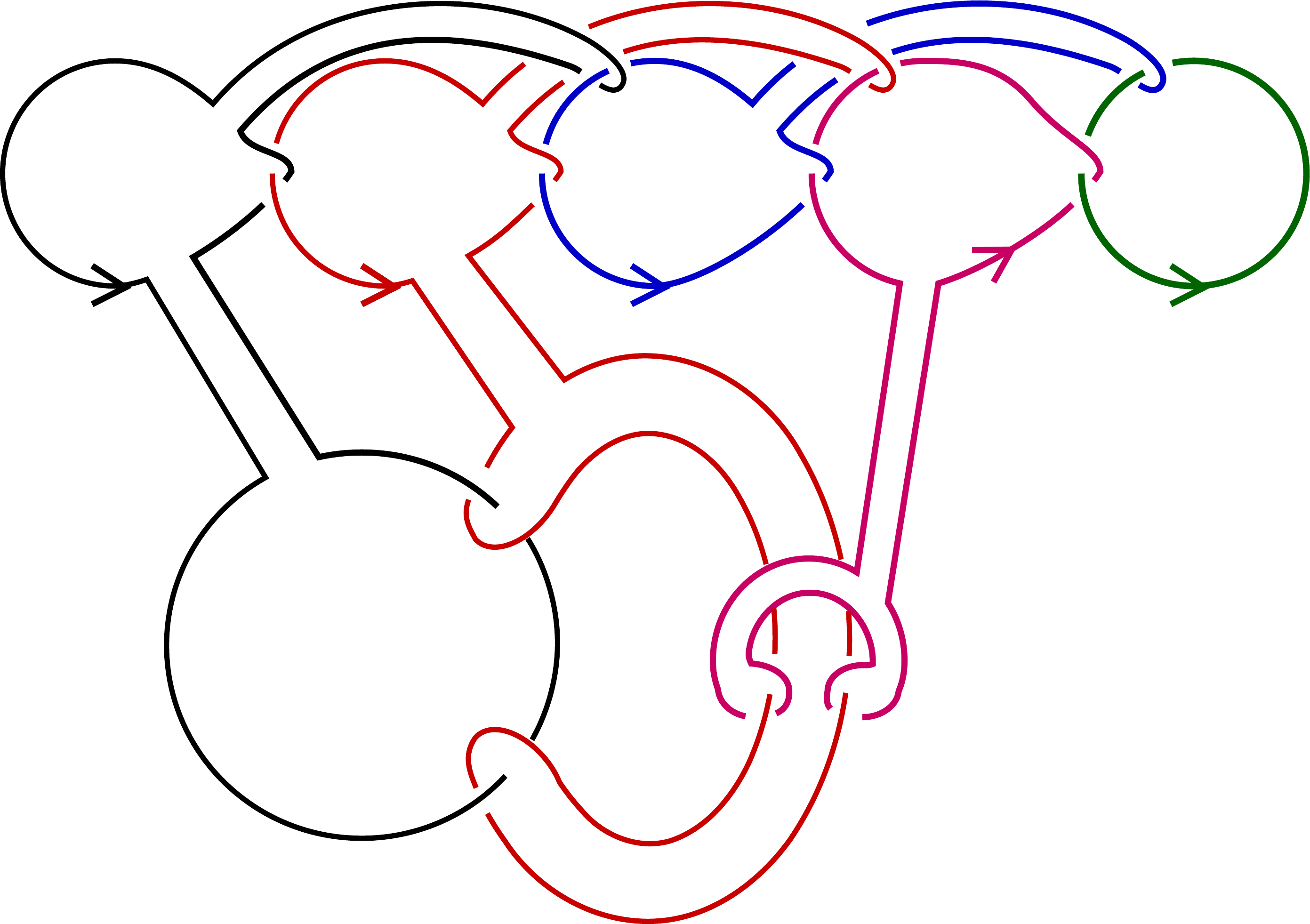}}
\end{picture}
\caption{It is a surprising a consequence of Theorems \ref{thm: trivial M} and \ref{thm:DNOP} that these two links admit homeomorphic surface systems.  }
\label{fig: trivial M}
\end{figure}

Conversely, we see that when $n\ge 6$, for any fixed $n\times n$ linking matrix $\Lambda$, $\M(\Lambda)$ is an infinite group by Theorem \ref{thm: M rank}.  By Theorems \ref{thm: trivial M} and \ref{thm:DNOP} there exist infinitely many links with linking matrix $\Lambda$ but no pair of which admit homeomorphic surface systems.   

There is an alternative approach to the non-triviality of $\M$.  In Section~\ref{sect: computation} we explain how to present $\M\otimes \Z/2$ in terms of the pairwise linking numbers mod 2.    As there are only finitely many $n\times n$ symmetric matrices with entries in $\Z/2$ and zeros on the main diagonal, one can list them one at a time and compute the rank (over $\Z/2$) of $\M\otimes \Z/2$ in each of these cases.  Such a computation reveals the following theorem. 

\begin{theorem}\label{thm: computer}
\begin{enumerate}
\item
Amongst all $2^{6}$ of the $4\times 4$ linking matrices containing only $0$'s and $1$'s, the following table summarizes the $(\Z/2)$-ranks of the resulting Total Milnor quotients:

$$
\begin{array}{|c|c|c|c|c|c|c|c|c|c|c|c|}
\hline
\rank_{\Z/2}(\M\otimes \Z/2)&0&1&2&3&4
\\\hline
\text{number of occurrences}&36&21&6&0&1
\\\hline
\text{portion of total}&.56&.33&.09&0&.02
\\\hline
\end{array}
$$
\item
Amongst all $2^{10}$ of the $5\times 5$ linking matrices containing only $0$'s and $1$'s, the following table summarizes the $(\Z/2)$-ranks of the resulting Total Milnor quotients:
$$
\begin{array}{|c|c|c|c|c|c|c|c|c|c|c|c|}
\hline
\rank_{\Z/2}(\M\otimes \Z/2)&0&1&2&3&4&5&6&7&8&9&10
\\\hline
\text{occurrences}&132&450&180&165&46&40&0&10&0&0&1
\\\hline
\text{portion of total}& .13& .44& .18& .16& .04& .04& 0& .01&0& 0& $.001$

\\\hline
\end{array}
$$
\item \label{item: n=6}
Amongst all $2^{15}$ of the $6\times 6$ linking matrices containing only $0$'s and $1$'s, the following table summarizes the ranks of the resulting Total Milnor quotients:
$$
\begin{array}{|c|c|c|c|c|c|c|c|c|c|c|c|}
\hline
\rank_{\Z/2}(\M\otimes \Z/2)&0&1&2&3&4&5&6
\\\hline
\text{occurrences}&0&0&5712&7920&8595&4035&1627
\\\hline
\text{portion of total}&0&0& .17& .24& .26& .12& .05
\\\hline
\hline
\rank_{\Z/2}(\M\otimes \Z/2)&7&8&9&10&11&12&13
\\\hline
\text{occurrences}&3030&855&240&538&75&45&80
\\\hline
\text{portion of total}& .09& .03& .007& 0.02&.002& .001& .002
\\\hline
\hline
\rank_{\Z/2}(\M\otimes \Z/2)&14&15&16&17&18&19&20
\\\hline
\text{occurrences}&0&0&15&0&0&0&1
\\\hline
\text{portion of total}&0&0& .0005&0&0&0& 3\times 10^{-5}
\\\hline
\end{array}
$$
\end{enumerate}

\end{theorem}

In particular, conclusion \pref{item: n=6} of Theorem~\ref{thm: computer} implies that $\M\otimes \Z/2$, and so $\M$, is nonzero for every $6\times 6$ linking matrix.  In order to see that this implies Theorem~\ref{thm: nontrivial M} we need the result proposition indicating that $\M$ grows when new components are added to a link.

\begin{proposition}\label{prop:sublink}
Let $L'$ be a sublink of $L$.  Let $\Lambda$ and $\Lambda'$ be their linking matrices.  Let $\M(\Lambda')$ and $\M(\Lambda)$ be the resulting Milnor quotients.  There is a surjection $\M(\Lambda)\onto \M(\Lambda')$.
\end{proposition}

The appearance of zeros on the far left of the table in conclusion \pref{item: n=6} of Theorem~\ref{thm: computer} imply Theorem~\ref{thm: nontrivial M}.  The astute reader will notice that these tables also contain some surprising zeros on the far right.  It would be interesting to see if this pattern is indicative of anything for links of arbitrarily many components.  We ask the following question:

\begin{question}
Consider any $n\in \N$.  Are there any $n$-components links for which $\M\otimes \Z/2$ has $\Z/2$-rank equal to ${n \choose 3}-1$?  
More generally, what abelian groups can by realized as total Milnor quotients of $n$-component links?
\end{question}

%

\subsection{Organization of the paper}

In Section~\ref{sect:background} we recall the precise definition of the total Milnor quotient and total triple linking number.  In passing we prove Proposition \ref{prop:sublink}.  In Section \ref{sect:realizability} we explain how the value of the total triple linking number changes under the operation of tying in a copy of the Boromean rings, in doing so we prove Theorem~\ref{thm: linking numbers}.  In Section~\ref{sect: nontrivial} we find lower bounds on the rank of the total Milnor quotient, proving Theorem \ref{thm: M rank} and as a consequence Theorem~\ref{thm: nontrivial M}.  In Section~\ref{sect: trivial} we construct a $5\times 5$ linking matrix for which $\M$ is trivial, proving Theorem~\ref{thm: trivial M}.  In Section~\ref{sect: computation} we explain the computations needed to prove Theorem~\ref{thm: computer} and explain why it also implies Theorem~\ref{thm: nontrivial M}.

\section{Background: surface systems, triple linking numbers, and the total Milnor quotient}\label{sect:background}

 In this section we provide some background.  We begin by recalling precisely what we mean by a  surface system for a link.  Next we quickly state the formulation of the triple linking number of a link in terms of a surface system due to Mellor-Melvin \cite{Mellor-Melvin}, which generalizes some of the ideas of \cite{Cochran90}.  Finally we give the definition of the total Milnor quotient $\M$ and total triple linking number $\mu(L)\in \M$ appearing in \cite{DNOP}.

For a link $L = L_1\cup\dots\cup L_n$, a \emph{surface system} $F = F_1\cup\dots\cup F_n$ for $L$ is a union of embedded surfaces with $\bdry F_i = L_i$, which are allowed to intersect transversely.  A pair of surface systems $F = F_1\cup\dots\cup F_n$ and $F' = F_1'\cup\dots\cup F_n'$ are called \emph{homeomorphic} if there exist a homeomorphism $\Phi:F\to F'$ which restricts to an orientation preserving homeomorphism $\Phi|_{F_i}:F_i\to F'_i$ and preserves orientations of the various intersection submanifolds.


We explain how to compute Milnor's triple linking number using a technique due to Mellor-Melvin \cite{Mellor-Melvin}.  Let $F = F_1\cup\dots \cup F_n$ be a surface system for the link $L$.  For each component, $F_k$, pick a basepoint $p_k\in L_k= \bdry F_k$.  Follow $L_k$ starting at $p_k$ and record the intersection with the components of $F$ as a word.  More precisely, whenever $L_k$ intersects $F_j$ positively record the letter $x_j$, when $L_k$ intersects $F_j$ negatively record $x_j^{-1}$.  In \cite{Davis-Roth}, This is called a \emph{clasp-word}, and is denoted $w_k(F)$. To each $i,j,k$ we associate a number $\epsilon_{ijk}(F)$ counting how many times an $x_i$ occurs before an $x_j$ in $w_k(F)$ as in the following definition.  

\begin{definition}[Definition 5.1 of \cite{DNOP}]\label{defn}
Let $w = x_{t_1}^{\epsilon_1} x_{t_2}^{\epsilon_2}\dots x_{t_m}^{\epsilon_m}$ be a word in $x_1, x_1^{-1} \dots, x_n, x_n^{-1}$.  An $(x_ix_j)$-decomposition of $w$ consists of a pair $(p,q)$, with $1\le p<q\le m$, $t_p = i$, and $t_q=j$.  The sign of this decomposition is $\sign_w(p,q) = \epsilon_p\cdot\epsilon_q$.  We define $\epsilon_{ijk}(F)$ by 
$$\epsilon_{ijk}(F) = \Sum_{p,q} \sign_{w_k(F)}(p,q)$$
where the sum is taken over all $(x_ix_j)$-decompositions of $w_k(F)$.  
\end{definition}

Set $m_{ijk}(F)= \epsilon_{ijk}(F)+\epsilon_{jki}(F)+\epsilon_{kij}(F)$ and $t_{ijk}(F)$ to be the signed count of the points in $F_i\cap F_j\cap F_k$.  Mellor-Melvin  \cite{Mellor-Melvin} prove that the difference recovers Milnor's triple linking number. That is, $\mu_{ijk}(L)\in \Z/\GCD(\lk(L_i,L_j), \lk(L_i,L_k), \lk(L_j,L_i))$ is equal to the class of $m_{ijk}(F)-t_{ijk}(F)$.   
 See \cite[Section 2]{Mellor-Melvin},  for an example computing the triple linking number using this perspective.  When the linking numbers $\lk(L_i,L_j)$, $\lk(L_i,L_k)$, and $\lk(L_j,L_i)$ do not all vanish, the value of $m_{ijk}(F)-t_{ijk}(F)$ depends on the choice of basepoints $p_i, p_j, p_k$ and on the choice of surface system.  In this case, $\mu_{ijk}(L)$ is only well-defined modulo the GCD of the linking numbers.  

Suppose $L$ and $L'$ admit homeomorphic surface systems $F\cong F'$.  Recall that the linking number $\lk(L_i, L_j)$ can be computed by counting the intersections between $L_i$ and $F_j$.  Since a homeomorphism of surface systems will preserve these intersection points, it follows that $\lk(L_i,L_j) = \lk(L_i',L_j')$.  In the case of 2-component links, \cite[Theorem 1]{Davis-Roth} gives the converese, two 2-component links admit homeomorphic surface systems if and only if they have the same linking numbers.

Similarly for the triple linking number, if you make a choice of basepoints on $F$ and $F'$ which are related by this homeomorphism, then you will find that the resulting clasp-words are identical.  Similarly the number of triple intersection points will be preserved.  As Mellor-Melvin's formula for $\mu_{ijk}$ involves only these claspswords and triple points, $\mu_{ijk}(L) = \mu_{ijk}(L')$, Thus, if two links admit homeomorphic surface systems, then their triple linking numbers will agree.  
In \cite[Theorem 2]{Davis-Roth} Roth and the third author prove that two links with vanishing pairwise linking numbers admit homeomorphic surface systems if and only if they have the same triple linking numbers.  

In \cite{DNOP} Nagel, Orson, Powell, and the third author gather together all $n\choose3$ triple linking numbers into a single element of $\Z^{n\choose3}$, called the {total triple linking number}.  They define a quotient $\M$ of $\Z^{n\choose3}$ where the class of this element is a well defined invariant of $L$.

\begin{definition}[Definition 5.6 of \cite{DNOP}]\label{defn:M}
Let $\Lambda$ be an $n\times n$ linking matrix.  Let $\{X^{ijk}~:~1\le i<j<k\le n\}$ be a basis for the alternating tensor $\Z^n\wedge\Z^n\wedge\Z^n\cong \Z^{n\choose3}$.  Let $V$ be the subspace of $\Z^{n\choose3}$ generated by $\{v_{jk}:1\le j\neq k\le n\}$ where $v_{jk} = \Sum_{i=1}^n \Lambda_{ik}X^{ijk}$.  The \emph{total Milnor quotient} associated with $\Lambda$ is the quotient group $\M:=\Z^{n\choose3}/ V$ .  
\end{definition}

For completeness we state some properties of the basis elements $X^{ijk}$ of the alternating tensor.   If any two of $i,j,k$ are the same then $X^{ijk}=0$.  For any $i,j,k$, $X^{ijk} = X^{jki} = X^{kij} =  -X^{ikj} = -X^{jik} = -X^{kji}$.  That is, the basis element $X^{ijk}$ is preserved under an even permutation of the indices and changed by a sign by an odd permutation.   In general, $v_{jk}$ and $v_{kj}$ are not even linearly dependent.  

A momentary reflection reveals that $\M$ is presented by $n\choose 3$ generators and $2\cdot {n\choose 2}$ relators.  When $n=9$, ${n\choose 3} =84$ while $ 2\cdot {n\choose 2} = 72$, so that by the rank-nullity theorem, $\rank(\M)\ge 12$ and $\M$ is an infinite abelian group.  See also \cite[Remark 5.10]{DNOP}.  

In \cite{DNOP} the \emph{total triple linking number} is defined by fixing a surface system $F$ for $L$, fixing a choice of basepoints, and considering all $n\choose 3$  triple linking numbers at once as an element of $\M$.  That is, 
\begin{equation}\label{total triple linking}\mu(L):=\Sum_{1\le i<j<k\le n}\left(m_{ijk}(F)-t_{ijk}(F)\right)X^{ijk}\in \M.\end{equation}
  In \cite[Corollary 1.4]{DNOP}, they not only show that this is an invariant of $L$, but that it determines whether two links admit homeomorphic surface system.   Any two links $L$ and $L'$ admit homeomorphic surface system if and only if the have the same linking matrix $\Lambda$ (so that they have the same total Milnor quotient $\M(\Lambda)$) and $\mu(L)=\mu(L')$ in $\M(\Lambda)$.  We now have all of the background needed.  We close this section with the proof of Proposition~\ref{prop:sublink}.

\begin{repproposition}{prop:sublink}
Let $L'$ be a sublink of $L$.  Let $\Lambda$ and $\Lambda'$ be their linking matrices.  Let $\M(\Lambda')$ and $\M(\Lambda)$ be the resulting Milnor quotients.  There is a surjection $\M(\Lambda)\onto \M(\Lambda')$.
\end{repproposition}
\begin{proof}

Let $L$ be a link with linking matrix $\Lambda$.  Without loss of generality, we assume that $L'$ is the $(n-1)$-component link given by deleting the $n$'th component of $L$.  Then $\Lambda'$ is the result of deleting the last row and column from $\Lambda$.  Since this proposition makes reference to two different linking matrices and two different total Milnor quotients, we will use $v^{\Lambda}_{jk}$ and $v^{\Lambda'}_{jk}$ to denote the generators of the subspaces $V(\Lambda)\subseteq \Z^{n\choose3}$ and $V(\Lambda')\subseteq \Z^{{n-1}\choose3}$ given in Definition~\ref{defn:M}.  The resulting total Milnor quotients are given by $\M(\Lambda) = { \Z^{n\choose3}}/{V(\Lambda)}$ and $\M(\Lambda') = { \Z^{{n-1}\choose3}}/{V(\Lambda')}$.

Consider the epimorphism $\Psi:  \Z^{{n}\choose3} \onto  \Z^{{n-1}\choose3}$ sending $\Psi(X^{ijk}) = X^{ijk}$ if $i$, $j$, and $k$ are all less than $n$ and sending $\Psi(X^{ijk}) = 0$ otherwise.  In order to show that $\Psi$ passes to a well-defined epimorphism from $\M(\Lambda)$ to $\M(\Lambda')$, it suffices to check that each generator $v_{jk}^{\Lambda}$ of $V(\Lambda)$ is sent to an element of the span of $V(\Lambda')$.  This follows from a quick inspection.  First suppose that one of $j$ and $k$ is equal to $n$ so that $\Psi(X^{ijk})=0$ for all $i$.  Then
$$
\Psi(v_{jk}^\Lambda) = \Sum_k\Lambda_{ij} \Psi(X^{ijk}) = 0
$$
which is certainly in the subspace $V(\Lambda')$.  Next suppose that each of $j$ and $k$ is less than $n$.  As $\Psi(X^{ijk}) = X^{ijk}$ when $i<n$ and $ \Psi(X^{ijk})=0$ when $i=n$, 
$$
\Psi(v_{jk}^\Lambda) = \Sum_{i=1}^n\Lambda_{ij} \Psi(X^{ijk}) = \Sum_{i=1}^{n-1}\Lambda_{ij} X^{ijk}.
$$
  Finally, as $\Lambda'$ is a submatrix of $\Lambda$, $\Lambda_{ij} =\Lambda_{ij}'$.  Thus,
$$
\Psi(v_{jk}^\Lambda) =  \Sum_{i=1}^{n-1}\Lambda'_{ij} X^{ijk}.
$$
This is precisely the definition of $v_{jk}^{\Lambda'}\in V(\Lambda')$.  Linearity now implies that $\Psi[V(\Lambda)]\subseteq V(\Lambda')$ so that $\Psi$ induces a well defined surjection $\M(\Lambda)\onto \M(\Lambda')$, completing the proof.
\end{proof}

\section{Realizability of triple linking numbers}\label{sect:realizability}

Let $\Lambda$ be any linking matrix and $\M(\Lambda)$ be the resulting total Milnor quotient.  In order to show that any element $m \in \M(\Lambda)$ is realized by some link with linking matrix $\Lambda$ we will show that the coefficient $m_{ijk}$ in $\mu(L) = \Sum_{ijk}m_{ijk} X^{ijk}$  may be incerased or decreased by $1$ by banding an appropriate copy of the Borromean Rings as in Figure~\ref{fig:changeMilnor}.

\begin{figure}[h]
\begin{picture}(340,70)
\put(0,10){\includegraphics[width=6em]{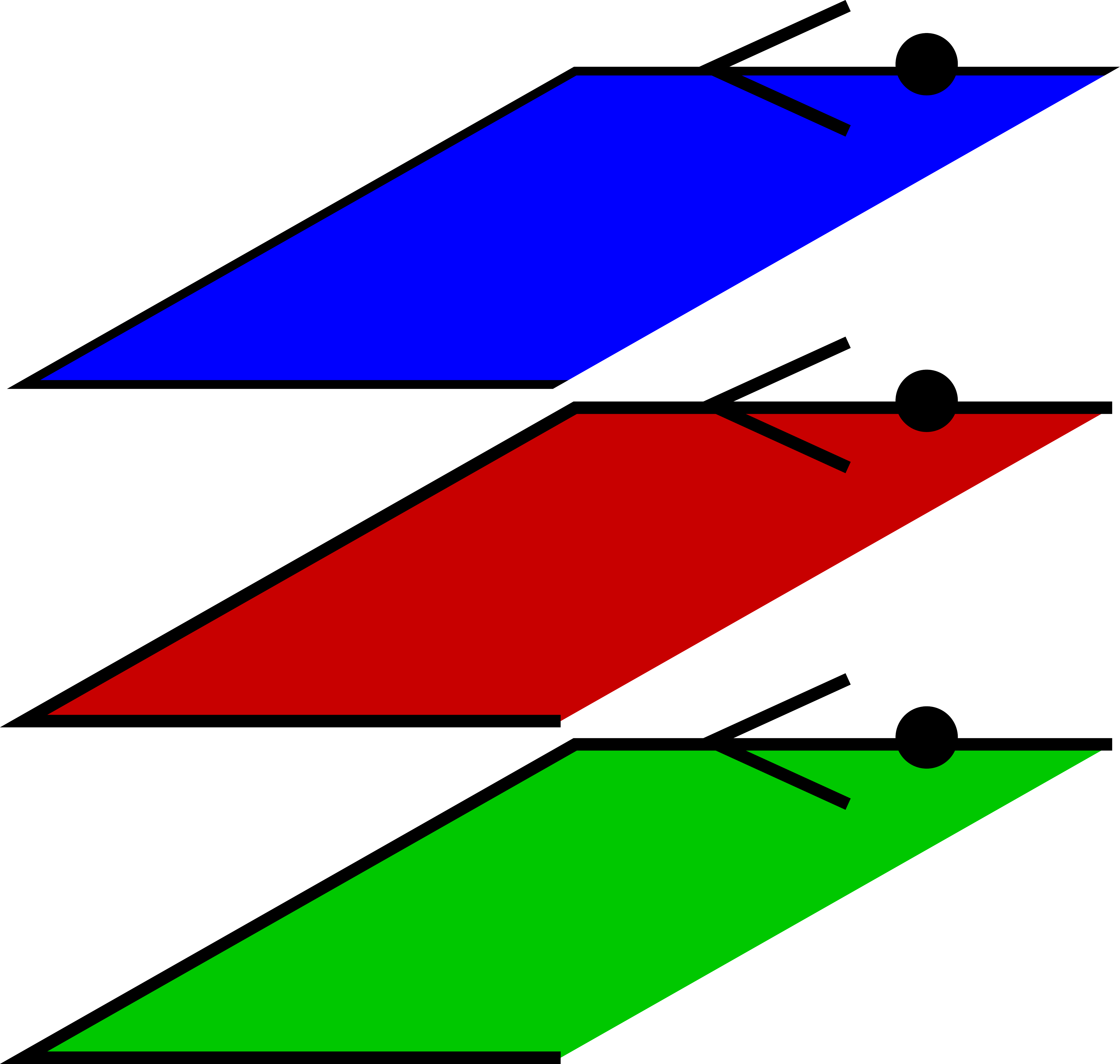}}
\put(60,58){\textcolor{blue}{\small{$F_i$}}}
\put(60,38){\textcolor{red}{\small{$F_j$}}}
\put(60,18){\textcolor{OliveGreen}{\small{$F_k$}}}
\put(80,0){\includegraphics[width=12em]{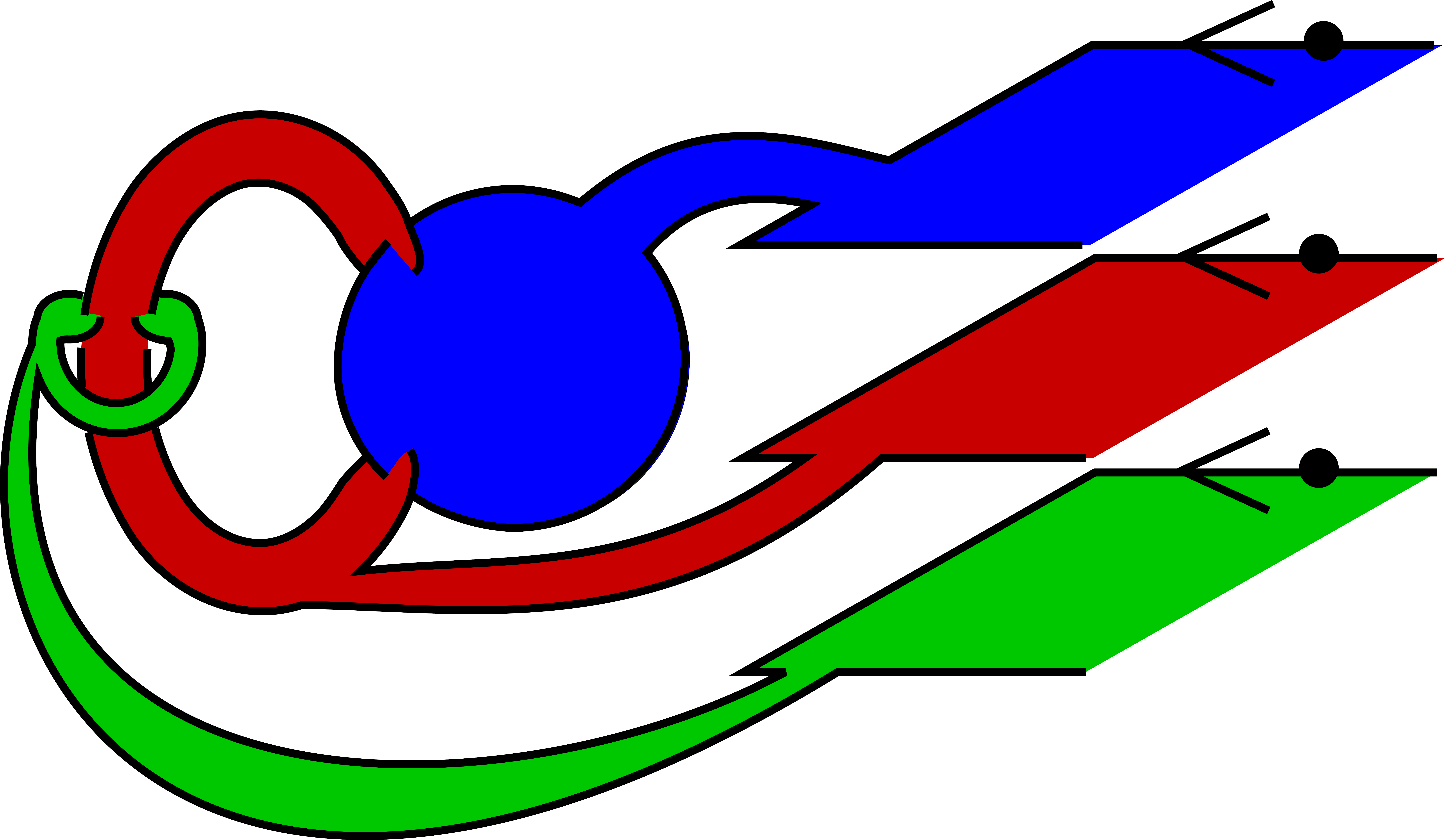}}
\put(200,60){\textcolor{blue}{\small{$F_i^+$}}}
\put(82, 51){\tiny{$-$}}
\put(96, 51){\tiny{$+$}}
\put(115, 60){\tiny{$+$}}
\put(105, 35){\tiny{$-$}}
\put(200,40){\textcolor{red}{\small{$F_j^+$}}}
\put(200,20){\textcolor{OliveGreen}{\small{$F_k^+$}}}
\put(225,0){\includegraphics[width=12em]{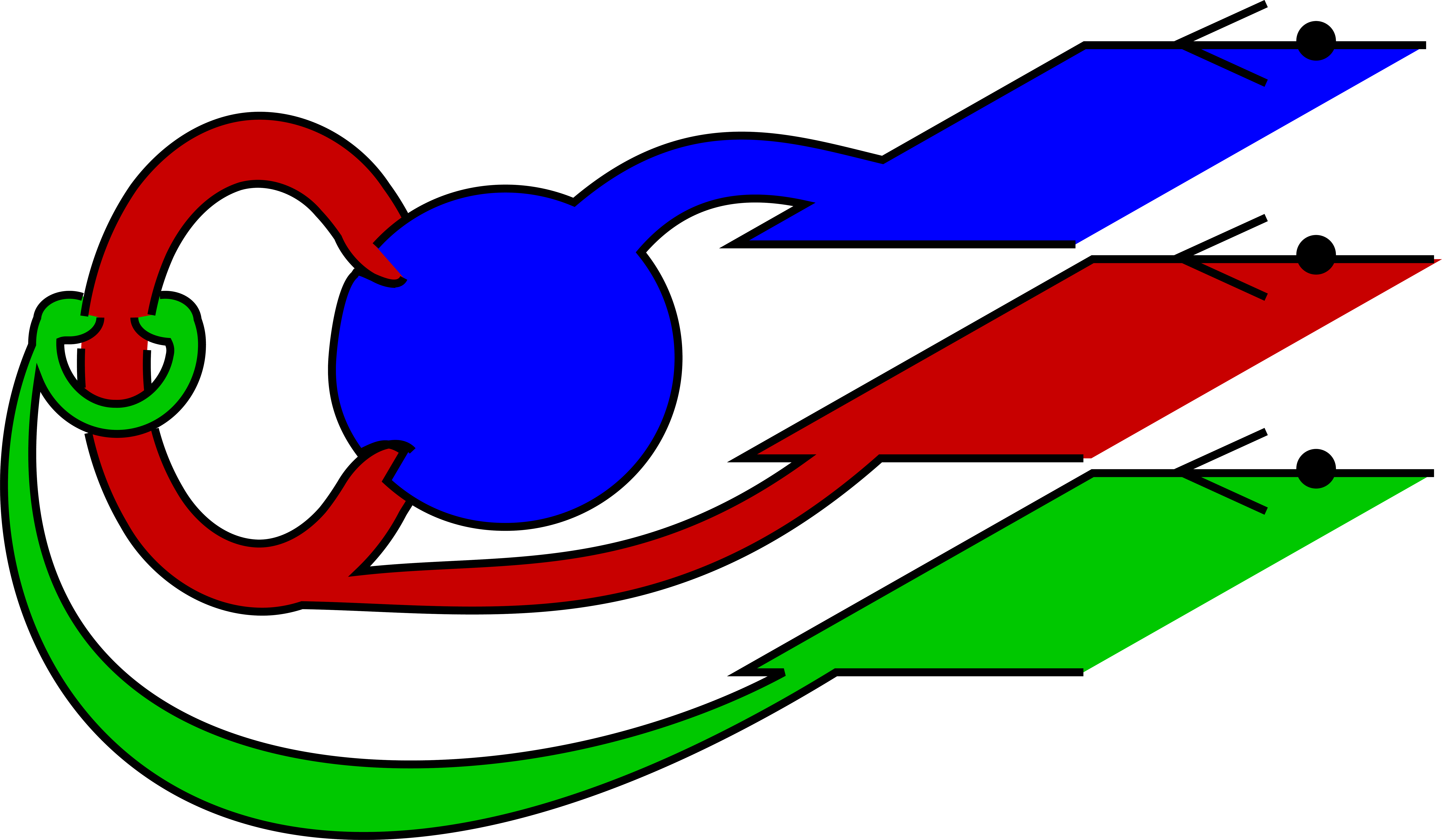}}
\put(227, 51){\tiny{$-$}}
\put(241, 51){\tiny{$+$}}
\put(260, 60){\tiny{$-$}}
\put(250, 35){\tiny{$+$}}
\put(345,60){\textcolor{blue}{\small{$F_i^-$}}}
\put(345,40){\textcolor{red}{\small{$F_j^-$}}}
\put(345,20){\textcolor{OliveGreen}{\small{$F_k^-$}}}
\end{picture}
\caption{Left:  A local picture of the components $F_i$, $F_j$, and $F_k$ of a surface system $F$ for a link $L$.  Center and right:  surface systems $F^+$ and $F^-$ for new links $L^+$ and $L^-$.  The $\pm$ next to each clasp indicates the sign the intersection points at either end.   
}
\label{fig:changeMilnor}
\end{figure}
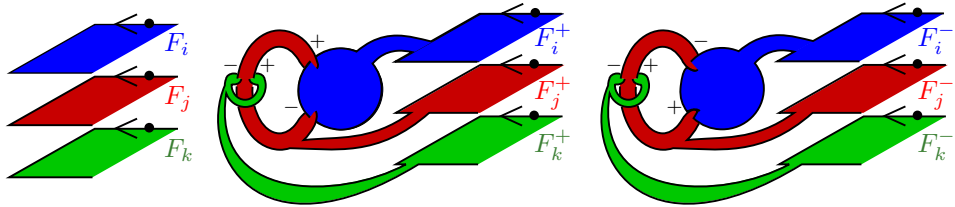

\begin{lemma}\label{lem:ChangeMilnor}
Let $L$ be any $n$-component link with linking matrix $\Lambda$.  Fix any $i,j,k$ with $1\le i<j<k\le n$.  Then the links $L^+$ and $L^-$ of Figure \ref{fig:changeMilnor} also have linking matrix $\Lambda$ and in the total Milnor quotient $\M(\Lambda)$, $\mu(L^+) = \mu(L)+X^{ijk}$ and $\mu(L^-) = \mu(L)-X^{ijk}$.
\end{lemma}
\begin{proof}
Let $F$ be a surface system for $L$ with base points $p_1,\dots, p_n$.  Let $w_m$ be the clasp-word given by following $L_m = \bdry F_m$.  We shall analyze how these clasp-words change when one replaces $F$ by either of the new surface systems $F^+$ or $F^-$ of Figure~\ref{fig:changeMilnor}.  We choose base points on $L^{+}=\bdry F^+$ and $L^-=\bdry F^-$ as in Figure~\ref{fig:changeMilnor}.  Let $w_m^+$ and $w_m^-$ denote the clasp-words given by following $L_m^+$ and $L_m^-$ respectively. Notice that for $m\notin\{ i,j,k\}$ no new clasps have been added to $\bdry F_m$ and so $w_m^-= w_m^+= w_m$.  Reading off the remaining clasp-words, 
$$
\begin{array}{ccc}
w_i^+ = x_jx_j^{-1} w_i,&
w_j^+ = x_i^{-1} x_k x_i x_k^{-1} w_i,&
w_k^+ = x_jx_j^{-1}w_k, 
\\
w_i^- = x_j^{-1}x_j w_i,&
w_j^- = x_i x_k x_i^{-1} x_k^{-1} w_i,&
w_k^- = x_jx_j^{-1}w_k .
\end{array}$$
We explain how $\epsilon_{kij}(F^+)$ differs from $\epsilon_{kij}(F)$.  The subword $ x_i^{-1} x_k x_i x_k^{-1}$ of $w_j^+$ includes one $(x_k, x_i)$-decomposition, namely $(2,3)$, whose  sign is $+1$.  Otherwise, $w_j^+$ contains all of the $(x_k, x_i)$-decompositions of $w_j$ and two new $(x_k,x_i)$-decompositions with opposite signs for every instance of $x_i$ in $w_j$.  Therefore $\epsilon_{kij}(F^+) = \epsilon_{kij}(F)+1$.  Similar analysis gives
$$
\begin{array}{ccc}
\epsilon_{ijk}(F^+) = \epsilon_{ijk}(F),&
\epsilon_{jki}(F^+) = \epsilon_{jki}(F),&
\epsilon_{kij}(F^+) = \epsilon_{kij}(F)+1,
\\
\epsilon_{ijk}(F^-) = \epsilon_{ijk}(F),&
\epsilon_{jki}(F^-) = \epsilon_{jki}(F),&
\epsilon_{kij}(F^-) = \epsilon_{kij}(F)-1.
\end{array}$$
Recall that  $m_{ijk}(F) = \epsilon_{ijk}(F) + \epsilon_{kij}(F) + \epsilon_{jki}(F)$.  Hence, 
$$m_{ijk}(F^+) = \epsilon_{ijk}(F^+) + \epsilon_{jki}(F^+) + \epsilon_{kij}(F^+) =  \epsilon_{ijk}(F) + \epsilon_{jki}(F) + (\epsilon_{kij}(F)+1) = m_{ijk}(F) + 1.$$
Similarly, $m_{ijk}(F^-)=m_{ijk}(F)-1$.  For any other $p<q<r$, $m_{pqr}(F^+) = m_{pqr}(F^-) = m_{pqr}(F)$.  The modification of Figure~\ref{fig:changeMilnor} introduces no new triple points, so that $t_{pqr}(F^+) = t_{pqr}(F^-) = t_{pqr}(F)$ for all $p<q<r$.  By equation \pref{total triple linking}, 
$$\mu(L^+) = \Sum (m_{pqr}(F^+)-t_{pqr}(F^+))X^{pqr}  = \Sum (m_{pqr}(F)-t_{pqr}(F))X^{pqr}  + X^{ijk}=\mu(L)  + X^{ijk},$$
where the sums are over all $1\le p<q<r\le n$.  Similarly
$\mu(L^-) = \mu(L)  - X^{ijk}$,  proving the result.  
\end{proof}

We are now ready to prove Theorem~\ref{thm: linking numbers}, which says every element of the total Milnor quotient is realized by some link.

\begin{reptheorem}{thm: linking numbers}
Let  $\Lambda$ be a linking matrix and $\M(\Lambda)$ be the resulting total Milnor quotient.  For every element $m \in \M(\Lambda)$ there exists a link, $L$, with linking matrix $\Lambda$ such that $\mu(L) = m$.
\end{reptheorem}

\begin{proof}[Proof of Theorem~\ref{thm: linking numbers}]

Let $\Lambda$ be an $n\times n$ linking matrix.  Let $\M$ be the resulting total Milnor quotient.  Consider any $m = \left[\Sum_{ijk}m_{ijk} X^{ijk}\right]\in \M$.  Consider now any $n$-component link $L$ with linking matrix $\Lambda$.  If $\mu(L)=m$ then we are already done.  Otherwise, $\mu(L) = \left[\Sum_{ijk}\mu_{ijk} X^{ijk}\right]\in \M$ and $\mu_{ijk}\neq m_{ijk}$ for some $1\le i<j<k\le n$.   If  $\mu_{ijk}<  m_{ijk}$, then we replace $L$ by $L^+$ as in Lemma \ref{lem:ChangeMilnor}.  The difference between $\mu_{ijk}$ and $m_{ijk}$ reduces by 1.  If $\mu_{ijk}> m_{ijk}$ then we instead use $L^-$.    Iterating, we replace $L$ by a new link with $\mu_{ijk} = m_{ijk}$.  Notice that this replacement leaves every other $\mu_{pqr}$ unchanged.  We repeat this procedure for every $i<j<k$ until we arrive at a final link $L$ with linking matrix $\Lambda$ and for which $\mu(L)=m$.  This completes the proof.
\end{proof}

\section{On the non-triviality of the total Milnor quotient}\label{sect: nontrivial}

Fix an $n\times n$ linking matrix $\Lambda$.  In this section we prove that $\M$ is nontrivial whenever $n$ is at least 6.  Recall that $\M$ is defined to be the quotient of $\Z^n\wedge \Z^n \wedge \Z^n = \Z^{n \choose 3}$ by the subspace $V$ spanned by $\{v_{jk}~:~1\le j\neq k\le n\}$ where $v_{jk} = \Sum_{i=1}^n \Lambda_{ik}X^{ijk}$.  By the rank-nullity theorem from linear algebra we see that 
$\rank(\M) = {n \choose 3} - \rank(V).$
In order to Prove Theorem~\ref{thm: nontrivial M}, we derive an upper bound on the rank of $V$ by finding linear dependencies amongst the $v_{jk}$, as in the following lemma.  

\begin{lemma}\label{lem: dependent V}
The subset $\{v_{jk}~:~1\le j\neq k\le n\}\subseteq \Z^{n \choose 3}$ admits the following linear dependencies:
\begin{enumerate}
\item \label{dependence:FixJ}For any fixed $1\le j\le n$, $\Sum_{k=1}^n v_{jk} = 0$
\item \label{dependence:FixK}For any fixed $1\le k\le n$, $\Sum_{j=1}^n \Lambda_{jk}v_{jk} = 0$
\end{enumerate}
\end{lemma}
\begin{proof}
The proof is a straightforward computation.  Indeed, fix any $1\le j\le n$.  Then
$$
\Sum_{k=1}^n v_{jk} = \Sum_{i=1}^n \Sum_{k=1}^n \Lambda_{ik}X^{ijk} = \Sum_{i,k} \Lambda_{ik}X^{ijk}.
$$
As $X^{iji}=0$, we drop the terms $i=k$ from the sum.  Thus, we may break this sum up into two sub-sums depending on whether $i<k$ or $k<i$. 
$$
\Sum_{k=1}^n v_{jk} =  \Sum_{i<k} \Lambda_{ik}X^{ijk} +  \Sum_{k<i} \Lambda_{ik}X^{ijk}.
$$
We reindex the latter sum by interchanging the roles of $i$ and $k$, 
$$
\Sum_{k=1}^n v_{jk} =  \Sum_{i<k} \Lambda_{ik}X^{ijk} +  \Sum_{i<k} \Lambda_{ki}X^{kji}.
$$
Now, as $\Lambda$ is symmetric and $X^{kji}=-X^{ijk}$, 
$$
\Sum_{k=1}^n v_{jk} =  \Sum_{i<k} \Lambda_{ik}X^{ijk} -  \Sum_{i<k} \Lambda_{ik}X^{ijk} = 0.
$$
This proves \pref{dependence:FixJ}.

Fix any $1\le k\le n$.  Then 
$\Sum_{j=1}^n \Lambda_{jk}v_{jk} = \Sum_{i,j} \Lambda_{jk}\Lambda_{ik}X^{ijk}$.
Similarly to before, we break into two sub-sums, reindex, and use that $X^{jik} = -X^{ijk}$,
$$\Sum_{j=1}^n \Lambda_{jk}v_{jk} = \Sum_{i<j} \Lambda_{jk}\Lambda_{ik}X^{ijk} +  \Sum_{j<i} \Lambda_{jk}\Lambda_{ik}X^{ijk} = \Sum_{i<j} \Lambda_{jk}\Lambda_{ik}X^{ijk} +  \Sum_{i<j} \Lambda_{ik}\Lambda_{jk}X^{jik} =0,$$
completing the proof.
\end{proof}

We are finally ready to prove Theorem~\ref{thm: M rank}.

\begin{reptheorem}{thm: M rank}
Let $n\ge 6$, $\Lambda$ be an $n\times n$ linking matrix, and $\M(\Lambda)$ be the resulting total Milnor quotient.  Then $\rank(\M(\Lambda))\ge\frac{n^3-9n^2+20n-6}{6}$. 
\end{reptheorem}

\begin{proof}

Let $n\ge 6$ and $\Lambda$ be an $n\times n$ linking matrix.  The proof will consist of two cases. In the first we assume that for every $k$, there exists some $j_k\neq k$ for which $\Lambda_{j_kk}\neq 0$.  

As a consequence of conclusion \pref{dependence:FixK} of Lemma \ref{lem: dependent V}, we see that we may realize a nonzero multiple of $v_{j_kk}$ as a linear combination of other generators of $V$,  $-\Lambda_{j_kk}v_{j_k,k} = \Sum_{j\neq j_k} \Lambda_{jk}v_{jk}$.  There is at most one value, call it $J$ if it exists, such that $j_k=J$ for all $k\neq J$.  Thus, for all $j\neq J$ there exists some $k$ with $j_k\neq j$.  We shall call this choice $k_j$.  Solving \pref{dependence:FixK} of Lemma \ref{lem: dependent V}, for $v_{jk_j}$ gives  $v_{jk_j} = -\Sum_{k \neq k_j} v_{jk}$.  Thus, the rank of $V$ is unchanged if we remove from its generating set, $\{v_{ij}~:~1\le i\neq j\le n\}$, each $v_{j_kk}$ where $k=1,\dots n$ and each $v_{jk_j}$ with $j\neq J$.    We have eliminated $2n-1$ of the original $2\cdot{n\choose 2}$ generators from $V$ without altering its rank and so $\rank(V)\le 2\cdot {n\choose 2} - 2n+1 = n^2-3n+1$.   As $\M = {\Z^{n \choose 3}}/{V}$, the rank-nullity theorem  implies 
 $$
 \rank(\M)
 ={n \choose 3} - \rank(V)
 \ge {n \choose 3} - n^2+3n-1 = \frac{n^3-9n^2+20n-6}{6},
 $$ 
 as Theorem~\ref{thm: M rank} asserts.  

It remains to deal with the case that there exists some $k$ such that $\Lambda_{ik}=0$ for all $i\neq k$.  In this case we see that for every $j\neq k$,  $v_{jk} = \Sum_i  \Lambda_{ik} X^{ijk} = 0$.  Thus, these $n-1$ generators of $V$ are all zero and so can be removed from the generating set.  For all $j$, let $k_j$ be different from $j$ and from $k$.  (This can be done as $n\ge 6> 3$).  Using conclusion \pref{dependence:FixK} of Lemma \ref{lem: dependent V} $v_{jk_j} = -\Sum_{k \neq k_j} v_{jk}$.  These $n$ generators of $V$ can be removed without changing $V$.  We have removed $2n-1$ elements from a generating set for $V$ without changing its rank. The proof now proceeds in the same manner as in the first case.   
\end{proof}

Theorem~\ref{thm: nontrivial M} is now an immediate consequence. 

\begin{reptheorem}{thm: nontrivial M}
Let $\Lambda$ be an $n\times n$ linking matrix with $n\ge 6$.  The resulting total Milnor quotient $\M(\Lambda)$ is non-trivial.  
\end{reptheorem}
\begin{proof}

  Factoring and completing the square on the lower bound on $\rank(\M)$ from Theorem~\ref{thm: M rank}, 
$$\frac{n^3-9n^2+20n-6}{6} = \frac{(n-3)\left((n-3)^2-7\right)}{6},$$ which is positive when $n>3+\sqrt{7} \approx 5.4$.  Thus, for all $n\ge 6$, $\rank(\M) > 1$, implying $\M$ is not the trivial group.
\end{proof}

\section{An example with trivial total Milnor quotient.}\label{sect: trivial}

We demonstrate that the $n\ge 6$ in Theorem~\ref{thm: nontrivial M} is sharp.  That is, we produce a $5\times 5$ linking matrix for which $\M$ is trivial. 

\begin{reptheorem}{thm: trivial M}
Let $\Lambda = \left[\begin{array}{ccccc}0&1&1&0&0\\1&0&1&1&0\\1&1&0&1&1\\0&1&1&0&1\\0&0&1&1&0
\end{array}\right]$.  The resulting total Milnor quotient $\M(\Lambda)$ is trivial.
\end{reptheorem}

\begin{proof}
Definition \ref{defn:M} says that $\M(\Lambda)$ is is the quotient of $\Z^{5\choose 3} = \Z^10$ by the subspace $V$ spanned by $\{v_{jk}:1\le j\neq k\le n\}$.  The latter set has  $2\cdot {5 \choose 2}=20$ generators.  These may be compiled into a $10\times 20$ presentation matrix, as in \cite[Section 8.4]{Rotman2010}.  One may use a computer put this matrix into Smith normal form \cite[Theorem 8.59]{Rotman2010} and see that the resulting matrix consists of $1$'s on the main diagonal and $0$'s elsewhere.  Thus, $\M$ is the trivial group.  Indeed, the choice of $\Lambda$ given in the proposition was produced by performing a computerized search and algorithmically putting a presentation matrix into Smith normal form.    

In place of  this computer driven approach, we indulge in a direct argument which can be verified without a machine.  Indeed, we shall realize each $X^{ijk}$ as linear combinations of the, $v_{jk} = \Sum_i \Lambda_{ik}X^{ijk}$.  
First, we expand out the definition of $v_{jk}$ for ten choices of $j$ and $k$.  During each expansion we use the fact that $X^{ijk}$ is preserved by even permutations of $(ijk)$ and negated by odd permutations we put each $(ijk)$ in increasing order.
\begin{equation}\label{eqn:v's}
\begin{array}{rclrcl}
v_{31}&=&X^{123},&v_{12}&=&X^{123}+X^{124},\\
v_{41}&=&X^{124}+X^{134},&v_{14}&=&-X^{124}-X^{134}+X^{145},\\
v_{15}&=&-X^{135}-X^{145},&v_{51}&=&X^{125}+X^{135},\\
v_{32}&=&-X^{123}-X^{234},&v_{24}&=&-X^{234}+X^{245},\\
v_{25}&=&-X^{235}-X^{245},&v_{45}&=&X^{345}.
\end{array}
\end{equation}
The reader will notice that if one reads these equations left-to-right, and then top-to-bottom, each equation introduces exactly one new $X^{ijk}$-term which does not appear in previous equations.  Thus, it is possible to solve for the various $X^{ijk}$.  For the sake of being explicit, we do so:
\begin{equation}\label{eqn:x's}
\begin{array}{rclrcl}
X^{123}&=&v_{31},&X^{124}&=&v_{12}-v_{31},\\
X^{134}&=&v_{41}-v_{12}+v_{31},&X^{145}&=&v_{14}+v_{41},\\
X^{135}&=&-v_{15}-v_{14}-v_{41},&X^{125}&=&v_{51}+v_{15}+v_{14}+v_{41},\\
X^{234}&=&-v_{32}-v_{31},&X^{245}&=&v_{24}-v_{32}-v_{31},\\
X^{235}&=&v_{32}+v_{31}-v_{25}-v_{24},&X^{345}&=&v_{45}.\\
\end{array}
\end{equation}
A thorough reader will now check that the equations in \pref{eqn:v's} imply those in \pref{eqn:x's}.  As each generator $X^{ijk}$ of $\Z^{n\choose 3}$ is zero in $\M(\Lambda)$, it follows that $\M(\Lambda)=0$.
\end{proof} 

\section{A computational approach to the nontriviality of the total Milnor quotient}\label{sect: computation}

Section~\ref{sect: nontrivial} provides a proof of the nontriviality of the total Milnor quotient for links of at least six components.  In this section we shall explore a different proof of the same result.  We will explain how to reduce the proof of Theorem \ref{thm: nontrivial M} to a finite number of computations, each of which can be done by a machine.  First we prove that the nonexistence of a link of six components with trivial total Milnor quotient implies the same for links of at least six components.

\begin{proposition}
If every link of six components has nontrivial trivial total Milnor quotient, thenevery link of at least six components has nontrivial total Milnor quotient.
\end{proposition}
\begin{proof}
Let $L$ be a link of at least six components with linking matrix $\Lambda$.  Let $L'$ be any 6-component sublink of $L$  and $\Lambda'$ be its linking matrix.  By Proposition \ref{prop:sublink}  there is a surjection $\M(\Lambda)\onto \M(\Lambda')$.  By the assumption of the proposition, $\M(\Lambda')\neq 0$, and so it follows that $\M(\Lambda)\neq 0$.  
\end{proof}

This reduces the proof of Theorem \ref{thm: nontrivial M} to an analysis of the total Milnor quotient associated to every $6\times 6$ linking matrix.  Now, there are infinitely many $6\times 6$ linking matrices, in order to side-step this we consider $\M\otimes \Z/2$.  By the right exactness of tensor product \cite[Theorem 6.113]{Rotman2010},
$$
\M\otimes \Z/2  = \left(\frac{\Z^{n \choose 3}}{V}\right)\otimes \frac{\Z}{2} \cong \frac{(\Z/2)^{n \choose 3}}{V\otimes \Z/2}.
$$
The tensored up subspace $V\otimes \Z/2$ is generated by 
$$v_{jk}\otimes 1 = \left( \Sum_{i=1}^n \Lambda_{ik}X^{ijk}\right)\otimes 1 =  \Sum_{i=1}^n [\Lambda_{ik}](X^{ijk}\otimes 1)$$
where $[\Lambda_{ij}]\in \Z/2$ is the result of reducing $\Lambda_{ij}$ mod 2.  Thus, $v_{jk}\otimes 1$, and so $V\otimes \Z/2$ depends only on the entires of $\Lambda$ reduced mod 2.  Thus, $\M\otimes \Z/2$ depends only in the entries of $\Lambda$ reduced mod 2.  

Therefore, in order to prove that $\M(\Lambda)$ is nontrivial for every choice of $6\times 6$ linking matrix $\Lambda$, it suffices to prove that $\M(\Lambda)\otimes \Z/2$ is nontrivial for every $6\times 6$ linking matrix $\Lambda$ with whose every entry is either 0 or 1.  Recall that a linking matrix is a symmetric matrix whose diagonal entries are zero.  There are only $2^{6 \choose 2} = 2^{15}$ such matrices.  Thus, if we can prove that all $2^{15}$ of these choices of $\Lambda$ result in $M(\Lambda)\otimes \Z/2$ nontrivial, we will be able to conclude that every 6-component link $L$ has nontrivial total Milnor quotient.

Definition \ref{defn:M} gives a presentation of $\M$ with ${6 \choose 3} = 20$ generators ($X^{ijk}$) and $2\cdot {6 \choose 2} = 30$ relators ($v_{jk}$).   Let $P$ be the resulting $20\times 30$ presentation matrix.  There exists a sequence of row and column moves reducing $P$ to a matrix $P'$ which is zero away form its main diagonal.  That is, we may put $P$ into \emph{Smith normal form} \cite[Theorem 8.59]{Rotman2010}.  These row and column moves preserve the presented group.  
If $P'$ has diagonal entries $d_1, d_2, d_3,\dots d_{20}$, then the group $\M$ is $\Oplus_{i=1}^{20}\Z/{d_i}$, where $\Z/0 = \Z$.  Thus, $\M\otimes \Z/2$ is given by $\Oplus_{i=1}^{20}\frac{\Z}{d_i}\otimes\frac{\Z}{2}\cong \Oplus_{i=1}^{20}\frac{\Z}{\GCD(2,d_i)}$ and the $\Z/2$-rank of $\M\otimes \Z/2$ is given by the number of $d_i$'s which are even.  A computer may now be used to loop through every $2^{6\choose 2} = 2^{15}$ symmetric $6\times 6$ matrix $\Lambda$ with zeros on the main diagonal, recover the resulting $20\times 30$ presentation matrix for $\M(\Lambda)$, put that presentation into Smith normal form and determine how many of the resulting diagonal entries are even.  The results of this census (as well as the census for $n=4$ and $n=5$) are recorded in Theorem \ref{thm: computer}.

\bibliographystyle{plain}
\bibliography{biblio}  

\end{document}